\documentclass[reqno]{amsart}
\usepackage[english]{babel}
\usepackage{amssymb,verbatim}
\usepackage[T1]{fontenc}

\newtheorem{theorem}{Theorem}[section]
\newtheorem{proposition}[theorem]{Proposition}
\newtheorem{lemma}[theorem]{Lemma}
\newtheorem{corollary}[theorem]{Corollary}

\theoremstyle{definition}
\newtheorem{definition}[theorem]{Definition}

\theoremstyle{remark}
\newtheorem*{remark}{Remark}

\numberwithin{equation}{section}

\newcommand{\C}{\mbox{$\mathbb{C}$}}

\begin{document}

\author{Per \AA hag}
\address{Department of Mathematics and Mathematical Statistics\\ Ume\aa \ University\\ SE-901 87 Ume\aa \\ Sweden}
\email{Per.Ahag@math.umu.se}
\author{Rafa\l\ Czy{\.z}}
\address{Institute of Mathematics\\ Jagiellonian University\\ \L ojasiewicza 6\\ 30-348 Krak\'ow\\ Poland}
\email{Rafal.Czyz@im.uj.edu.pl}

\title[A characterization of the degenerate complex Hessian equations]{A characterization of the degenerate complex Hessian equations for functions with bounded $(p,m)$-energy}

\begin{abstract} By proving an estimate of the sublevel sets for $(\omega,m)$-subharmonic functions we obtain a Sobolev type inequality that is then used to characterize the degenerate complex Hessian equations for such functions with bounded $(p,m)$-energy.
\end{abstract}

\maketitle

\begin{center}\bf
\today
\end{center}

\section{Introduction}

Ever since the 1930s when the interest in K\"ahler geometry gained momentum with the publication of Erich K\"ahler's article~\cite{Kahler}, the attention has been
immense both from mathematicians and physicists. Take, for example, the works of Aubin~\cite{Aubin} and Yau~\cite{Yau}, as well as
the highly regarded Seiberg-Witten theory~\cite{SeibergWitten1, SeibergWitten2} of physics. In the mentioned work of Aubin and Yau they showed how
geometric information of a K\"ahler manifold can be retrieved by solving certain partial differential equation of  Monge-Amp\`ere type.  This is part of the explanation of why these equations and associated methods have been of great interest in recent decades. Our motivation is instead from a pluripotential theoretical background and the highly influential work of Bedford and Taylor~\cite{BedfordTaylor1,BedfordTaylor2}, and Cegrell~\cite{CegrellPCE,CegrellGDM}.

 Combing the ideas of Cegrell's energy classes with globally defined plurisubharmonic functions known as $\omega$-plurisubharmonic functions
 Guedj and Zeriahi introduced and studied weighted energy classes of $\omega$-plurisubharmonic functions (\cite{GZ07}). In particular, they proved the existence of solutions to the Dirichlet problem for the complex Monge-Amp\`ere operator, and later Dinew proved the uniqueness (\cite{Dinew2}).

  Here we shall also use the idea of energy classes, but for the interpolation spaces of $m$-subharmonic functions. These spaces interpolate between subharmonic and plurisubharmonic functions, and the differential operator is the complex Hessian operator. The idea of these interpolation spaces goes back to Caffarelli et al.~\cite{CNS85}, and pluripotential methods were introduced by B\l ocki in~\cite{B}.

 The general setting of this paper is that $n\geq 2$, $p>0$, and  $1\leq m\leq n$. Furthermore, we shall use $(X,\omega)$ to denote a connected and compact K\"ahler manifold of complex dimension $n$, where $\omega$ is a K\"ahler form on $X$ such that $\int_{X}\omega^n=1$. The energy classes of $(\omega,m)$-subharmonic functions with bounded $(p, m)$-energy that is central for this paper is defined by
 \[
\mathcal E_m^p(X,\omega):=\left\{u\in \mathcal E_m(X,\omega): u\leq 0, e_{p,m}(u)<\infty\right\},
\]
where
\[
e_{p,m}(u)=\int_X(-u)^p\operatorname{H}_m(u),
\]
and $H_m$ denote the complex Hessian operator (see Section~\ref{sec_preliminaries} for details). For a historical account and references see e.g.~\cite{AC2019,LN2}.

\medskip

By proving in Lemma~\ref{cap} an estimate of the sublevel sets for $(\omega,m)$-subharmonic functions we obtain the following Sobolev type inequality.

\medskip

\noindent\textbf{Theorem~\ref{lemmaep}.} \emph{Let $n\geq 2$, $p>0$, and let $1\leq m\leq n$. Assume that $(X,\omega)$ is a connected and compact K\"ahler manifold of complex dimension $n$, where $\omega$ is a K\"ahler form on $X$ such that $\int_{X}\omega^n=1$. Furthermore, assume that $\mu$ is a Borel measure defined on $X$. Fix a constant $\beta$ such that $1>\beta>\max\Big(\frac {pn-n}{pn-n+m},\frac {p}{p+1}\Big)$, for $p>1$, and $\beta=\frac {p}{p+1}$ for $p\leq 1$. The following conditions are then equivalent:}
\begin{enumerate}\itemsep2mm
\item $\mathcal E_m^p(X,\omega)\subset L^q(X,\mu)$;
\item \emph{there exists a constant $C>0$ such that for all $u\in \mathcal E_m(X,\omega)\cap L^{\infty}(X)$ with $\sup_Xu=-1$ it holds}
\[
\int_X(-u)^q\,d\mu\leq Ce_{p,m}(u)^{\frac {q\beta}{p}};
\]
\item \emph{there exists a constant $C>0$ such that for all $u\in \mathcal E_m^p(X,\omega)$ with $\sup_Xu=-1$ it holds}
\[
\int_X(-u)^q\,d\mu\leq Ce_{p,m}(u)^{\frac {q\beta}{p}}.
\]
\end{enumerate}

\medskip

Theorem~\ref{lemmaep} is then used in Theorem~\ref{Dirichlet} to characterize the degenerate complex Hessian equation for $(\omega,m)$-subharmonic
functions with bounded $(p,m)$-energy. This equation was first considered for smooth solutions, and later for continuous functions (see e.g.~\cite{DK,LN,LN2} and references  therein). In~\cite{LN}, Lu and Nguyen \ recently solved the Dirichlet problem for the complex Hessian equation in $\mathcal E^1_m(X,\omega)$. In their paper, they used the variational method. By instead using our Sobolev type inequality we can in Theorem~\ref{Dirichlet} generalize Lu and Nguyen results to $p>0$.

\medskip

\noindent\textbf{Theorem~\ref{Dirichlet}.} \emph{Let $n\geq 2$, $p>0$, and let $1\leq m\leq n$. Assume that $(X,\omega)$ is a connected and compact K\"ahler manifold of complex dimension $n$, where $\omega$ is a K\"ahler form on $X$ such that $\int_{X}\omega^n=1$. Furthermore, assume that $\mu$ is a Borel probability measure defined on $X$. The following conditions
are then equivalent:}
\begin{enumerate}\itemsep2mm
\item $\mathcal E_m^p(X,\omega)\subset L^p(X,\mu)$;
\item \emph{there exists unique $(\omega,m)$-subharmonic function $u$ in $\mathcal E_m^p(X,\omega)$ such that $\sup_X u=-1$ and $\operatorname{H}_m(u)=\mu$.}
\end{enumerate}

\section{Preliminaries}\label{sec_preliminaries}

 Let $\Omega \subset \C^n$, $n\geq 2$, be a bounded domain, $1\leq m\leq n$, and define $\mathbb C_{(1,1)}$ to be the set of $(1,1)$-forms with constant coefficients. We then define
\[
\Gamma_m=\left\{\alpha\in \mathbb C_{(1,1)}: \alpha\wedge \beta^{n-1}\geq 0, \dots , \alpha^m\wedge \beta ^{n-m}\geq 0   \right\}\, ,
\]
where $\beta=dd^c|z|^2$ is the canonical K\"{a}hler form in $\C^n$.

\begin{definition}\label{m-sh} Let $n\geq 2$, and $1\leq m\leq n$. Assume that $\Omega \subset \C^n$ is a bounded domain, and let $u$ be a subharmonic function defined on $\Omega$. Then we say that $u$ is \emph{$m$-subharmonic} if the following inequality holds
\[
dd^cu\wedge\alpha_1\wedge\dots\wedge\alpha_{m-1}\wedge\beta^{n-m}\geq 0\, ,
\]
in the sense of currents for all $\alpha_1,\ldots,\alpha_{m-1}\in \Gamma_m$. With $\mathcal{SH}_m(\Omega)$ we denote the set of all $m$-subharmonic functions defined on $\Omega$.
\end{definition}

Let $\sigma_{k}$ be $k$-elementary symmetric polynomial of $n$-variable, i.e.,
\[
\sigma_{k}(x_1,\ldots,x_n)=\sum_{1\leq j_1<\cdots<j_k\leq n}x_{j_1}\cdots x_{j_k}\, .
\]
It can be proved that a function $u\in\mathcal C^2(\Omega)$ is $m$-subharmonic if, and only if,
\[
\sigma_k(u(z))=\sigma_k(\lambda_1(z),\ldots,\lambda_n(z))\geq 0,
\]
for all $k=1,\dots,m$, and all $z\in \Omega$. Here, $\lambda_1(z),\ldots,\lambda_n(z)$ are the eigenvalues of the complex Hessian matrix $\left [\frac {\partial ^2u}{\partial z_j\partial \bar z_k}(z)\right]$.

Next we shall consider compact K\"ahler manifold.

\begin{definition} Let $n\geq 2$, and let $1\leq m\leq n$. Assume that $(X,\omega)$ is a connected and compact K\"ahler manifold of complex dimension $n$, where $\omega$ is a K\"ahler form on $X$ such that $\int_{X}\omega^n=1$. A function $u: X\to \mathbb R\cup\{-\infty\}$ is called \emph{$(\omega,m)$-subharmonic} if in any local chart $\Omega$ of $X$, the function $f + u$ is $m$-subharmonic, where $f$ is a local potential of $\omega$. We shall denote by $\mathcal {SH}_m(X,\omega)$ the set of all $(\omega,m)$-subharmonic functions on $X$.
\end{definition}

 The following notation is convenient: for any $u\in \mathcal {SH}_m(X,\omega)$ let
\[
\omega_u=dd^cu+\omega.
\]
With this notation we have that a smooth function $u$ is $(\omega,m)$-subharmonic if, and only if,
\[
\omega_u^k \wedge \omega^{n-k} \geq 0 ,\qquad \text{ for all } k=1,\ldots,m.
\]

 In the following proposition we list useful properties of $(\omega,m)$-subharmonic functions. For proofs see e.g.~\cite{LN} and the references therein.

\begin{proposition}\label{prop}
Let $(X,\omega)$ be a compact K\"ahler manifold. Then
\begin{enumerate}\itemsep2mm
\item if $u,v\in \mathcal {SH}_m(X,\omega)$, $t\in [0,1]$, then $tu+(1-t)v\in \mathcal {SH}_m(X,\omega)$;
\item if $u\in \mathcal {SH}_m(X,\omega)$, $t\in [0,1]$, then $tu\in \mathcal {SH}_m(X,\omega)$;
\item if $u,v\in \mathcal {SH}_m(X,\omega)$, then $\max(u,v)\in \mathcal {SH}_m(X,\omega)$;
\item if $u_j\in \mathcal {SH}_m(X,\omega)$, $j\in \mathbb N$ then $(\sup_j u_j)^*\in \mathcal {SH}_m(X,\omega)$. Here $(\,)^*$ denotes the upper semicontinuous regularization;
\item if $u\in \mathcal {SH}_m(X,\omega)$, then there exists a decreasing sequence $u_j\in \mathcal {SH}_m(X,\omega)\cap \mathcal C^{\infty}(X)$ such that $u_j\to u$, $j\to \infty$.
\end{enumerate}
\end{proposition}

Following the idea from~\cite{GZ07} one can define the complex Hessian operator for $(\omega,m)$-subharmonic through the following construction.  First assume that $u\in \mathcal {SH}_m(X,\omega)\cap L^{\infty}(X)$, then
\[
\operatorname{H}_m(u):=\omega_u^m\wedge \omega^{n-m},
\]
which is a non-negative (regular) Borel measure on $X$. For an arbitrary (not necessarily bounded) $(\omega,m)$-subharmonic function $u$ let $u_j=\max(u,-j)$ be the canonical approximation of $u$. Then define
\[
\operatorname{H}_m(u):=\lim_{j\to \infty}\chi_{\{u>-j\}}\operatorname{H}_m(u_j).
\]
The complex Hessian operator is then used to construct the following function class.

\begin{definition}
Let $\mathcal E_m(X,\omega)$ be the class of all $(\omega,m)$-subharmonic functions defined as
\[
\mathcal E_m(X,\omega)=\left\{u\in \mathcal {SH}_m(X,\omega): \int_X\operatorname{H}_m(u)=1 \right\}.
\]
\end{definition}
\begin{remark}
Note that $u\in \mathcal E_m(X,\omega)$ if, and only if,
\[
\operatorname{H}_m(u_j)(\{u_j\leq -j\})\to 0, \qquad \text{as } j\to \infty.
\]
Here, $u_j=\max (u,-j)$.
\end{remark}

Let us collect some properties of the class $\mathcal E_m(X,\omega)$. Proofs can be found in~\cite{DL}.

\begin{theorem}\label{cp}
Let $(X,\omega)$ be a compact K\"ahler manifold.
\begin{enumerate}\itemsep2mm
\item If $u,v\in \mathcal E_m(X,\omega)$, $t\in [0,1]$, then $tu+(1-t)v\in \mathcal E_m(X,\omega)$, and $\max (u,v)\in \mathcal E_m(X,\omega)$. In particular, if $u,v\in \mathcal {SH}_m(X,\omega)$, $u\in \mathcal E_m(X,\omega)$ and $u\leq v$, then $v\in \mathcal E_m(X,\omega)$.

\item If $u_j\in \mathcal E_m(X,\omega)$ is decreasing sequence converging to $u\in \mathcal E_m(X,\omega)$, then
$\operatorname{H}_m(u_j)$ converges weakly to $\operatorname{H}_m(u)$.

\item If $u,v\in \mathcal E_m(X,\omega)$, then
\[
\chi_{\{u<v\}}\operatorname{H}_m(\max (u,v))=\chi_{\{u<v\}}\operatorname{H}_m(v).
\]
In particular, if $\operatorname{H}_m(u)\geq \mu$, $\operatorname{H}_m(v)\geq\mu$ for some Borel measure $\mu$, then $\operatorname{H}_m(\max (u,v))\geq \mu$.

\item (The comparison principle) Let $T$ be a positive current of the type
\[
T=\omega_{\psi_1}\wedge \dots\wedge \omega_{\psi_k}\wedge\omega^{n-m},
\]
where $k<m$, and $\psi_1,\dots,\psi_k\in \mathcal E_m(X,\omega)$. Then for $u,v\in \mathcal E_m(X,\omega)$ it holds
\[
\int_{\{u<v\}}\omega_v^{m-k}\wedge T\leq \int_{\{u<v\}}\omega_u^{m-k}\wedge T.
\]

\item (The Dirichlet problem for the complex Hessian operator) Let $\mu$ be a probability measure on $X$ that does not charge $m$-polar sets. Then there exists a unique function $u\in \mathcal E_m(X,\omega)$ such that $\operatorname{H}_m(u)=\mu$.
\end{enumerate}
\end{theorem}

We shall be in need of the $m$-capacity defined on a compact K\"ahler manifold $X$.

\begin{definition} For any Borel set $A\subset X$ define the \emph{$m$-capacity} of $A$ as
\[
\operatorname{cap}_m(A):=\sup\left \{\int_A\operatorname{H}_m(u): u\in \mathcal {SH}_m(X,\omega), \, -1\leq u\leq 0\right\}.
\]
We say that a Borel set $A\subset X$ is $m$-polar if $\operatorname{cap}_m(A)=0$.
\end{definition}

\begin{proposition}\label{prop1}
If $u\in \mathcal E_m(X,\omega)$, $u\leq 0$, then for any $t<0$ it holds
\[
\operatorname{cap}_m(\{u<-t\})\leq \frac {C}{t},
\]
where the constant $C$ does not depend on $u$.
\end{proposition}

%

A central part of the proof of Lemma~\ref{lq} is the following estimate due to Dinew and Ko\l odziej~\cite{DK}.

\begin{lemma}\label{dk}
For any $1<\alpha<\frac {n}{n-m}$ there exits a constant $C(\alpha)>0$ such that for any Borel set $A\subset X$ it holds
\[
V(A)\leq C(\alpha)\operatorname{cap}_m(A)^{\alpha},
\]
where $V(A)=\int_A\omega^n$.
\end{lemma}

\section{Functions with bounded $(p,m)$-energy}

In this section we focus on $(\omega,m)$-subharmonic functions with bounded $(p,m)$-energy, and prove some necessary properties that is needed
for the rest of this paper.

\begin{definition}\label{defEpm} Let $n\geq 2$, $p>0$, and let $1\leq m\leq n$. Assume that $(X,\omega)$ is a connected and compact K\"ahler manifold of complex dimension $n$, where $\omega$ is a K\"ahler form on $X$ such that $\int_{X}\omega^n=1$. We define the class of \emph{$(\omega,m)$-subharmonic functions with bounded $(p,m)$-energy} as
\[
\mathcal E_m^p(X,\omega):=\left\{u\in \mathcal E_m(X,\omega): u\leq 0, e_{p,m}(u)<\infty\right\},
\]
where
\[
e_{p,m}(u)=\int_X(-u)^p\operatorname{H}_m(u).
\]
\end{definition}
\begin{remark}
It was proved in~\cite{GZ07,LN} (see also~\cite{CegrellPCE,Dinew}) that $u\in \mathcal E_m^p(X,\omega)$ if, and only if, $\sup_je_{p,m}(u_j)<\infty$, where $u_j=\max (u,-j)$ is the canonical approximation of $u$. Furthermore,
$e_{p,m}(u_j)\to e_{p,m}(u)$ as $j\to \infty$.
\end{remark}

\begin{lemma}\label{4}
Let $(X,\omega)$ be a compact K\"ahler manifold, and $p>0$. Furthermore, let $1\leq j\leq m$, and  let $T$ be a positive current of the type
\[
T=\omega_{\psi_1}\wedge \dots\wedge \omega_{\psi_{m-j}}\wedge\omega^{n-m},
\]
where $\psi_1,\dots,\psi_{m-j}\in \mathcal E_m(X,\omega)$. Then for any $u,v\in \mathcal E_m(X,\omega)\cap L^{\infty}(X)$, $u,v\leq 0$, it holds
\[
\int_X(-u)^p\omega_v^j\wedge T\leq 2^p \int_X(-u)^p\omega_u^j\wedge T+2^p\int_X(-v)^p\omega_v^j\wedge T.
\]
\end{lemma}
\begin{proof} Note that we have the following $\{u<-2s\}\subset \{u<v-s\}\cup\{v<-s\}$, and therefore  by Theorem~\ref{cp} (4)
\[
\int_{\{u<v-s\}}\omega_v^j\wedge T\leq \int_{\{u<v-s\}}\omega_u^j\wedge T\leq \int_{\{u<-s\}}\omega_u^j\wedge T,
\]
and then
\begin{multline*}
\int_X(-u)^p\omega_v^j\wedge T=p\int_0^{\infty}t^{p-1}(\omega_v^j\wedge T)(\{u<-t\})\,dt\\
=p2^p\int_0^{\infty}s^{p-1}(\omega_v^j\wedge T)(\{u<-2s\})\,ds\\
\leq p2^p\left(\int_0^{\infty}s^{p-1}(\omega_v^j\wedge T)(\{u<v-s\})\,ds+ \int_0^{\infty}s^{p-1}(\omega_v^j\wedge T)(\{v<-s\})\,ds\right)\\
\leq p2^p\left(\int_0^{\infty}s^{p-1}(\omega_u^j\wedge T)(\{u<-s\})\,ds+ \int_0^{\infty}s^{p-1}(\omega_v^j\wedge T)(\{v<-s\})\,ds\right)\\
=2^p\left(\int_X(-u)^p\omega_u^j\wedge T+\int_X(-v)^p\omega_v^j\wedge T\right).
\end{multline*}
\end{proof}

\begin{lemma}\label{7}
Let $(X,\omega)$ be a compact K\"ahler manifold, and $p>0$.  Let $T$ be a positive current of the type
\[
T=\omega_{\psi_1}\wedge \dots\wedge \omega_{\psi_{m-1}}\wedge\omega^{n-m},
\]
where $\psi_1,\dots,\psi_{m-1}\in \mathcal E_m(X,\omega)$. Then for any $u,v\in \mathcal E_m(X,\omega)\cap L^{\infty}(X)$ such that $u\leq v\leq 0$ it holds
\begin{equation}\label{eq3}
\int_X(-u)^p\omega_{v}\wedge T\leq (p+1)\int_X(-u)^p\omega_{u}\wedge T.
\end{equation}
In particular, if $u,v\in \mathcal E_m^p(X,\omega)$ are such that $u\leq v\leq 0$, then
\[
e_{p,m}(v)\leq (p+1)^me_{p,m}(u).
\]
\end{lemma}
\begin{proof} From Proposition~\ref{prop} it follows that we can assume that $u$ is smooth and $u\leq v<0$.

\medskip

\noindent\emph{Case 1: ($p\geq 1$).} We have
\begin{equation}\label{eq1}
\int_X(-u)^p\omega_{v}\wedge T=\int_X(-u)^p\omega\wedge T+\int_Xvdd^c(-u)^p\wedge T=I_1+I_2.
\end{equation}
Note that
\begin{multline}\label{eq2}
I_1=\int_X(-u)^p\omega\wedge T\leq \int_X(-u)^p\omega\wedge T+p\int_X(-u)^{p-1}du\wedge d^cu\wedge T\\ =\int_X(-u)^p\omega_u\wedge T,
\end{multline}
and
\[
dd^c(-u)^p=p(p-1)(-u)^{p-2}du\wedge d^cu-p(-u)^{p-1}dd^cu\geq -p(-u)^{p-1}dd^cu.
\]
The integral $I_2$ can be estimated as follows
\begin{multline}\label{eq4}
I_2=\int_Xvdd^c(-u)^p\wedge T\leq p \int_X(-v)(-u)^{p-1}dd^cu\wedge T\\ \leq p\int_X(-v)(-u)^{p-1}\omega_u\wedge T
\leq p\int_X(-u)^{p}\omega_u\wedge T.
\end{multline}
Combining the inequalities (\ref{eq1}), (\ref{eq2}) and (\ref{eq4}) we get (\ref{eq3}).

\medskip

\noindent\emph{Case 2: ($0<p<1$).} We have by (\ref{eq2})
\begin{multline*}
\int_X(-u)^p\omega_{v}\wedge T=\int_X(-u)^p\omega\wedge T+\int_X(-v)dd^c(-(-u)^p)\wedge T\\
\leq \int_X(-u)^p\omega\wedge T+\int_X(-v)\left[p(-u)^{p-1}\omega+dd^c(-(-u)^p)\right]\wedge T\\
\leq \int_X(-u)^p\omega\wedge T+\int_X(-u)\left[p(-u)^{p-1}\omega+dd^c(-(-u)^p)\right]\wedge T\\
=p\int_X(-u)^p\omega\wedge T+\int_X(-u)^p\omega_u\wedge T\leq (p+1)\int_X(-u)^p\omega_u\wedge T.
\end{multline*}

The last statement of this lemma follows from the canonical approximation, and inequality (\ref{eq3}) applied $m$-times
\begin{multline*}
e_{p,m}(v)=\int_{X}(-v)^p\omega_v^m\wedge \omega^{n-m}\leq \int_{X}(-u)^p\omega_v^m\wedge \omega^{n-m}\\
\leq (p+1)\int_{X}(-u)^p\omega_u\wedge \omega_v^{m-1}\wedge \omega^{n-m}\leq \dots\leq (p+1)^m\int_{X}(-u)^p\omega_u^m\wedge \omega^{n-m}\\
=(p+1)^me_{p,m}(u).
\end{multline*}
\end{proof}

\begin{corollary}\label{cor1}
Let $(X,\omega)$ be a compact K\"ahler manifold, and $p>0$. The following conditions are then equivalent
\begin{enumerate}\itemsep2mm
\item $u\in \mathcal {E}_m^p(X,\omega)$;
\item for any decreasing sequence $u_j\in \mathcal {E}_m^p(X,\omega)$, $u_j\searrow u$ we have
\[
\sup_je_{p,m}(u_j)<\infty;
\]
\item there exists a decreasing sequence $u_j\in \mathcal {E}_m^p(X,\omega)$, $u_j\searrow u$ such that
\[
\sup_je_{p,m}(u_j)<\infty.
\]
\end{enumerate}
\end{corollary}
\begin{proof}
The equivalence (1)$\Leftrightarrow$(3) follows from the remark after Definition~\ref{defEpm}, and implication (2)$\Rightarrow$(3) is immediate.  Finally, we prove (3)$\Rightarrow$(2). Assume that there exists a decreasing sequence $u_j\in \mathcal {E}_m^p(X,\omega)$, $u_j\searrow u$ such that
\[
\sup_je_{p,m}(u_j)<\infty,
\]
and let $v_j$ be any sequence decreasing to $u$. Then for any $j$ there exists $k_j$ such that $v_j\geq u_{k_j}$. Therefore by Lemma~\ref{7} the sequence $e_{p,m}(v_j)$ is also bounded.
\end{proof}

\begin{lemma}\label{lem7}
Let $(X,\omega)$ be a compact K\"ahler manifold, and $p>0$. Then there exists a constant $C>0$ such that for any $u_0,u_1,\dots,u_m\in \mathcal E_m^p(X,\omega)$ it holds
\begin{equation}\label{eq10}
\int_X(-u_0)^p\omega_{u_1}\wedge\dots\wedge\omega_{u_m}\wedge\omega^{n-m}\leq C\max_{j=1,\dots,m}e_{p,m}(u_j).
\end{equation}
\end{lemma}
\begin{proof} By using the canonical approximation we can assume without lost of generality that all functions $u_0,\dots u_m$  are bounded. For $T=\omega_{u_1}\wedge\dots\wedge\omega_{u_m}\wedge\omega^{n-m}$, Lemma~\ref{4} yields
\[
\int_X(-u_0)^p\omega_{u_1}\wedge T\leq 2^p\int_X(-u_0)^p\omega_{u_0}\wedge T+2^p\int_X(-u_1)^p\omega_{u_1}\wedge T.
\]
Therefore we can assume that $u_0=u_1$. Set $u=\epsilon \sum_{j=1}^{m}u_j$, where $\epsilon$ is a small positive constant that will be specified later. It is sufficient to estimate  integrals of the type $\int_X(-u_1)^p\omega_{u}^m\wedge\omega^{n-m}$, since
\begin{equation}\label{12}
\omega_{u}^m\wedge \omega^{n-m}\geq \epsilon^m \omega_{u_1}\wedge\dots\wedge\omega_{u_m}\wedge\omega^{n-m}.
\end{equation}
Again by using Lemma~\ref{4}
\[
\int_X(-u_1)^p\omega_{u}^m\wedge\omega^{n-m}\leq 2^pe_{p,m}(u_1)+2^pe_{p,m}(u),
\]
and
\begin{multline}\label{11}
e_{p,m}(u)=\int_{X}\left(-\epsilon \sum_{j=1}^{m}u_j\right)^p\omega_{u}^m\wedge\omega^{n-m}\leq \max(\epsilon^p,\epsilon)\sum_{j=1}^{m}\int_X(-u_j)^p\omega_{u}^m\wedge\omega^{n-m}\\
\leq \max(\epsilon^p,\epsilon)2^pm\left(\max_{j=1,\dots,m}e_{p,m}(u_j)+e_{p,m}(u)\right).
\end{multline}
Now take $\epsilon$ such that $1-2^pm\max(\epsilon,\epsilon^p)>\frac 12$, then by (\ref{12}) and (\ref{11}) we get
\begin{multline*}
\int_X(-u_1)^p\omega_{u_1}\wedge\dots\wedge\omega_{u_m}\wedge\omega^{n-m}\leq \epsilon^{-m}\int_X(-u_1)^p\omega_{u}^m\wedge\omega^{n-m}\\
\leq \frac{4^pm}{\epsilon^m(1-2^pm\max(\epsilon,\epsilon^p))}\max_{j=1,\dots,m}e_{p,m}(u_j)+2^p\epsilon^{-m}e_{p,m}(u_j) \\
\leq 2^{p+1}\epsilon^{-m}\max_{j=1,\dots,m}e_{p,m}(u_j).
\end{multline*}
\end{proof}
\begin{remark} Assume that the functions $u_j\in \mathcal E_m^p(X,\omega)$ are such that $\sup_X u_j=-1$, and $\sup_{j\in \mathbb N}e_{p,m}(u_j)<\infty$. Then
\begin{equation}\label{end of sec3}
u=\sum_{j=1}^{\infty}\frac {1}{2^j}u_j\in \mathcal E_m^p(X,\omega).
\end{equation}
By using Corollary~\ref{cor1} it is sufficient to construct a decreasing sequence of functions $v_j\in \mathcal E_m^p(X,\omega)$, $v_j\searrow u$, $j\to \infty$, such that $\sup_{j\in \mathbb N}e_{p,m}(v_j)<\infty$. Let us next define
\[
v_j=\sum_{k=1}^ja_ku_k, \ \text{ where } \ a_k=\frac {2^j}{2^k(2^j-1)}.
\]
Then $v_j\in \mathcal {SH}_m(X,\omega)$, $v_j\searrow u$, and by Lemma~\ref{lem7} we get
\begin{multline*}
e_{p,m}(v_j)=\int_X\left(-\sum_{k=1}^ja_k u_k\right)^p\omega_{v_j}^m\wedge\omega^{n-m} \\
\leq \sum_{k=1}^j\max \left(a_k,a_k^p\right)\sum_{k_{1}+\dots+k_{1}=m}\binom{m}{k_1\dots k_j}a_1^{k_1}\cdots a_j^{k_j} \int_X(-u_k)^p\omega_{u_{1}}^{k_1}\wedge \dots\wedge \omega_{u_{j}}^{k_j}\wedge \omega^{n-m}\\
\leq \sum_{k=1}^j\max \left(a_k,a_k^p\right)\sum_{k_{1}+\dots+k_{1}=m}\binom{m}{k_1\dots k_j}a_1^{k_1}\cdots a_j^{k_j} C\max_{k=1,\dots,j}e_{p,m}(u_k) \\
\leq C\sup_{k\in \mathbb N}e_{p,m}(u_k),
\end{multline*}
which means that $v_j\in \mathcal E_m^p(X,\omega)$, and $\sup_{j\in \mathbb N}e_{p,m}(v_j)<\infty$. Thus,~(\ref{end of sec3}) holds.
\end{remark}

\section{A Sobolev type inequality}

The aim of this section is to prove the Sobolev type inequality in Theorem~\ref{sobolev}. We shall first need to prove the estimates of the sublevel sets for $(\omega,m)$-subharmonic functions with bounded $(p,m)$-energy in Lemma~\ref{thm_lq}.

\begin{lemma}\label{cap}
If $u\in \mathcal E_m(X,\omega)$, $t\in [0,1]$, $s>0$ then
\begin{equation}\label{1}
t^m\operatorname{cap}_m(\{u<-s-t\})\leq \int_{\{u<-s\}}\operatorname{H}_m(u)\leq s^m\operatorname{cap}_m(\{u<-s\}).
\end{equation}
Furthermore, if $u\in \mathcal E_m^p(X,\omega)$, $p>0$ and $s>1$, then
\begin{equation}\label{2}
\operatorname{cap}_m(\{u<-s\})\leq (s-1)^{-p}e_{p,m}(u).
\end{equation}
\end{lemma}
\begin{proof}
Let $v\in \mathcal E_m(X,\omega)$ be such that $-1\leq v\leq 0$. Then for $t\in [0,1]$ we get that $tv\in \mathcal E_m(X,\omega)$. Note that we have
\[
\{u<-s-t\}\subset \{u<-s+tv\}\subset \{u<-s\},
\]
and therefore by Theorem~\ref{cp} (4)
\begin{multline*}
\int_{\{u<-s-t\}}\operatorname{H}_m(v)\leq \int_{\{u<-s+tv\}}\operatorname{H}_m(v)\leq t^{-m}\int_{\{u<-s-t\}}\operatorname{H}_m(s+tv)\\
\leq t^{-m}\int_{\{u<-s+tv\}}\operatorname{H}_m(u)\leq t^{-m}\int_{\{u<-s\}}\operatorname{H}_m(u).
\end{multline*}
This proves the left inequality in (\ref{1}).

To prove the right inequality in (\ref{1}) we assume for a moment that $u$ is continuous and let $1\leq s<s_0$. Then
\begin{multline*}
\int_{\{u<-s\}}\operatorname{H}_m(\max (u,-s_0))=-\int_{\{u\geq -s\}}\operatorname{H}_m(\max (u,-s_0))+\int_{X}\operatorname{H}_m(\max (u,-s_0))\\
=-\int_{\{u\geq-s\}}\operatorname{H}_m(u)+\int_{X}\operatorname{H}_m(u)=\int_{\{u<-s\}}\operatorname{H}_m(u).
\end{multline*}
Note that $\frac {1}{s_0}\max (u,-s_0)\in \mathcal E_m(X,\omega)$, and $-1\leq \frac {1}{s_0}\max (u,-s_0)\leq 0$, and therefore
\begin{multline*}
\operatorname{cap}_m(\{u<-s\})\geq \int_{\{u<-s\}}\operatorname{H}_m\left(\frac {1}{s_0}\max (u,-s_0)\right) \\
\geq s_0^{-m}\int_{\{u<-s\}}\operatorname{H}_m(\max (u,-s_0))=s_0^{-m}\int_{\{u<-s\}}\operatorname{H}_m(u).
\end{multline*}
If $s_0\searrow s$, then we get
\[
\operatorname{cap}_m(\{u<-s\})\geq s^{-m}\int_{\{u<-s\}}\operatorname{H}_m(u).
\]
For the general situation take a smooth decreasing sequence $u_j\searrow u$ and observe that
\begin{multline*}
\int_{\{u<-s\}}\operatorname{H}_m(u)\leq \liminf_{j\to \infty}\int_{\{u_j<-s\}}\operatorname{H}_m(u_j)\\
\leq s^m \liminf_{j\to \infty} \operatorname{cap}_m(\{u_j<-s\})\leq s^m \operatorname{cap}_m(\{u<-s\}).
\end{multline*}

To prove inequality (\ref{2}) assume that $u\in \mathcal E_m^p(X,\omega)$, and $s>1$. We shall use (\ref{1}) to obtain
\[
t^m\operatorname{cap}_m(\{u<-s-t\})\leq \int_{\{u<-s\}}\operatorname{H}_m(u)\leq s^{-p}\int_{\{u<-s\}}(-u)^p\operatorname{H}_m(u)\leq s^{-p}e_{p,m}(u).
\]
For $t=1$ we get the desired conclusion.
\end{proof}

\begin{lemma}\label{lq}
If $u\in \mathcal E_m^p(X,\omega)$, then $u\in L^q(X)$ for $0<q<\frac {\max (p,1)n}{n-m}$.
\end{lemma}
\begin{proof}
Let $u\in \mathcal E_m^p(X,\omega)$. Assume first that $p\geq 1$. From Lemma~\ref{dk} and Lemma~\ref{cap} we have for $\alpha<\frac {n}{n-m}$
\begin{multline*}
\int_{X}(-u)^q\omega^n\leq 2^q+q\int_{2}^{\infty}t^{q-1}V(\{u<-t\})\,dt\\ \leq 2^q+qC(\alpha)\int_2^{\infty}t^{q-1}\left((t-1)^{-p}e_{p,m}(u)\right)^{\alpha}\,dt\\
=2^q+qC(\alpha)e_{p,m}(u)^{\alpha}\int_{2}^{\infty}t^{q-1}(t-1)^{-p\alpha}\,dt.
\end{multline*}
The right hand side is finite if, and only if, $q<p\alpha<\frac {pn}{n-m}$.

For $p<1$ we use Proposition~\ref{prop1} to obtain, in a similar way as above, that
\[
\int_{X}(-u)^q\omega^n\leq 2^q+q\int_{2}^{\infty}t^{q-1}(t-1)^{-\alpha}\,dt.
\]
The right hand side is finite if, and only if, $q<\alpha<\frac {n}{n-m}$.
\end{proof}

We shall need the following elementary fact.

\begin{proposition}\label{elementary}
Let $\alpha>0$ and $F:[0,\infty)\to [0,\infty)$  a decreasing function such that
\[
\int_{0}^{\infty}t^{\alpha}F(t)\,dt<\infty.
\]
Then there exists a constant $C>0$ such that for all $t>0$ we have $F(t)\leq Ct^{-\alpha-1}$.
\end{proposition}

\begin{proof}
Using integration by parts we get
\begin{multline*}
C=\int_0^{\infty}t^{\alpha}F(t)\,dt=\int_0^st^{\alpha}F(t)\,dt+\int_s^{\infty}t^{\alpha}F(t)\,dt\\
=\frac {s^{\alpha+1}F(s)}{\alpha+1}-\frac {1}{\alpha+1}\int_0^st^{\alpha+1}F'(t)\,dt+\int_s^{\infty}t^{\alpha}F(t)\,dt\geq \frac {s^{\alpha+1}F(s)}{\alpha+1}.
\end{multline*}
\end{proof}
\begin{remark}
From Lemma~\ref{lq} and Proposition~\ref{elementary} it follows that for all $u\in \mathcal E_m^p(X,\omega)$ there exists a constant $C(u,q)$ depending only on $u$ and $q$ such that
\begin{equation}\label{3}
V(\{u<-t\})\leq \frac {C(u,q)}{t^q}, \ \text{ for } \ 0<q<\frac {\max(p,1)n}{n-m}.
\end{equation}
\end{remark}

In Theorem~\ref{thm_lq} we prove estimates of the sublevel sets of $(\omega,m)$-subharmonic functions with bounded $(p,m)$-energy. For the
case $p=1$, Theorem~\ref{thm_lq} gives sharper estimates than those proved in~\cite{LN}.

\begin{theorem}\label{thm_lq} Let $n\geq 2$, and let $1\leq m\leq n$. Assume that $(X,\omega)$ is a connected and compact K\"ahler manifold of complex dimension $n$, where $\omega$ is a K\"ahler form on $X$ such that $\int_{X}\omega^n=1$. If $u\in \mathcal E_m^p(X,\omega)$, then
\begin{enumerate}\itemsep2mm
\item there exists a constant $C(u)$ depending only on $u$ such that for all $t>1$
\[
\operatorname{cap}_m(\{u<-t\})\leq \frac {C(u)}{t^{p+1}};
\]
\item there exists a constant $C(u,q)$ depending only on $u$ and $q$ such that for all $t>1$, and $0<q<\frac {(p+1)n}{n-m}$,
\[
V(\{u<-t\})\leq \frac {C(u,q)}{t^{q}};
\]
\item for all  $0<q<\frac {(p+1)n}{n-m}$,   we have that $u\in L^q(X)$.
\end{enumerate}
\end{theorem}
\begin{proof}
By Lemma~\ref{lq} we know that if $u\in \mathcal E_m^p(X,\omega)$, then $u\in L^q(X)$ for $0<q<\frac {\max(p,1)n}{n-m}$. Fix $u\in \mathcal E_m^p(X,\omega)$, $u\leq -1$, and $v\in \mathcal E_m(X,\omega)$, $-1\leq v\leq 0$, and let $t\geq 1$. Then $\frac {u}{t}\in \mathcal E_m(X,\omega)$, and
\[
\{u<-2t\}\subset \left\{\frac ut<v-1\right\}\subset \{u<-t\}.
\]
By Theorem~\ref{cp} (4) we obtain
\begin{multline*}
\int_{\{u<-2t\}}\omega_v^m\wedge\omega^{n-m}\leq \int_{\left\{\frac ut<v-1\right\}}\omega_v^m\wedge\omega^{n-m}\leq \int_{\left\{\frac ut<v-1\right\}}\omega_{\frac ut}^m\wedge\omega^{n-m}\\
\leq \int_{\{u<-t\}}\omega_{\frac ut}^m\wedge\omega^{n-m}\leq\int_{\{u<-t\}}\omega^n+t^{-1}\sum_{j=1}^n\binom{m}{j}\int_{\{u<-t\}}\omega_u^j\wedge\omega^{n-j}.
\end{multline*}
Now observe that by Lemma~\ref{4}
\[
\int_{\{u<-t\}}\omega_u^j\wedge\omega^{n-j}\leq t^{-p}\int_X(-u)^p\omega_u^{j}\wedge\omega^{n-j}\leq\frac {2^p}{t^p}e_{p,m}(u),
\]
and by (\ref{3})
\[
\int_{\{u<-t\}}\omega^n=V(\{u<-t\})\leq \frac {C(u,q)}{t^q}, \ \text { for } \ 0<q<\frac {\max(p,1)n}{n-m}.
\]
Therefore we get
\begin{equation}\label{5}
\operatorname{cap}_m(\{u<-t\})\leq \frac {C_1(u,q)}{t^{\min (p+1,q)}}, \ \text { for } q<\frac {\max (p,1)n}{n-m}\ .
\end{equation}
Let $p_1=\min (p+1,q)$, where $q<\frac {\max (p,1)n}{n-m}$. Then by (\ref{5}) we have
\[
\operatorname{cap}_m(\{u<-t\})\leq \frac {C_2(u,q)}{t^{p_1}}.
\]
By the proof of Lemma~\ref{lq} we get that $u\in L^{q_1}$, where $q_1=p_1\frac {n}{n-m}$, and then it follows from (\ref{3}) that
\[
V(\{u<-t\})\leq \frac {C_3(u,q_1)}{t^{q_1}}.
\]
Now we can once more repeat the argument above and obtain
\[
\operatorname{cap}_m(\{u<-t\})\leq \frac {C_4(u,q_1)}{t^{p_2}}e_{p,m}(u),
\]
where $p_2=\min (p+1,q_1)$. It would again imply that $u\in L^{q_2}(X)$ and
\[
V(\{u<-t\})\leq \frac {C_5(u,q_2)}{t^{q_2}},
\]
where $q_2<p_2\frac {n}{n-m}$. We can repeat this argument $l$-times until $p_l=p+1$ (which is possible since $\frac {n}{n-m}>1$). Finally, we get
\[
\operatorname{cap}_m(\{u<-t\})\leq \frac {\tilde C(u)}{t^{p+1}},
\]
$u\in L^{q}(X)$, and
\[
V(\{u<-t\})\leq \frac {\tilde C(u,q)}{t^{q}}, \ \text { for} \ 0<q<\frac {(p+1)n}{n-m}.
\]
\end{proof}

Now we can prove a Sobolev type inequality for $(\omega,m)$-subharmonic functions with bounded $(p,m)$-energy.

\begin{theorem}\label{sobolev} Let $X$ be a connected and compact K\"ahler manifold of complex dimension $n$, where $\omega$ is a K\"ahler form on $X$ such that $\int_{X}\omega^n=1$. Also let $1\leq m\leq n$ and $p\geq 1$. Then for any $1<q<\frac {pn}{n-m}$, and any $\epsilon >0$, there exists constant $C(\epsilon)$ such that for any $u\in \mathcal E_m^p(X,\omega)$, $\sup_X u=-1$, we have that
\[
\int_X(-u)^q\omega^n\leq C(\epsilon)e_{p,m}(u)^{\frac {n(q-1)}{np-n+m}+\epsilon}\, .
\]
\end{theorem}
\begin{proof}
Take $u\in \mathcal E_m^p(X,\omega)$, $\sup_X u=-1$, and fix $q<\frac {pn}{n-m}$. Also, let $q<Q<p\alpha$, where $1<\alpha<\frac {n}{n-m}$. Then we have by Lemma~\ref{cap}
\[
\begin{aligned}
&\int_X(-u)^q\omega^n=q\int_0^{\infty}t^{q-1}V(\{u<-t\})\,dt\\
&\leq q\left(\int_0^{\infty}t^{Q-1}V(\{u<-t\})\,dt\right)^{\frac{q-1}{Q-1}}\left(\int_0^{\infty}V(\{u<-t\})\,dt\right)^{\frac{Q-q}{Q-1}}\\
&\leq q\left(\frac {2^Q}{Q}+\int_2^{\infty}t^{Q-1}e_{p,m}(u)^{\alpha}(t-1)^{-p\alpha }\,dt\right)^{\frac{q-1}{Q-1}}\left(\int_X(-u)\omega^n\right)^{\frac{Q-q}{Q-1}}\\
&\leq Ce_{p,m}(u)^{\alpha\frac{q-1}{Q-1}}\left(\int_X(-u)\omega^n\right)^{\frac{Q-q}{Q-1}}.
\end{aligned}
\]
It follows from~\cite{L2} that if $u$ is $(\omega,m)$-subharmonic function such that $\sup_Xu=-1$, then there exists a constant $C'$ that does not depending on $u$ such that $\sup_j\int_X(-u)\omega^n\leq C'$.
Note that
\[
\inf_{\alpha,Q}\left\{\left(\alpha\frac{q-1}{Q-1}\right): q<Q<p\alpha \text{ and } 1<\alpha<\frac {n}{n-m}\right\}=\frac {n(q-1)}{np-n+m}\, .
\]
Therefore, for any $\epsilon >0$ there exists constant $C(\epsilon)$ that does not depending on $u$ such that
\[
\int_X(-u)^q\omega^n\leq C(\epsilon)e_{p,m}(u)^{\frac {n(q-1)}{np-n+m}+\epsilon}\, .
\]

\end{proof}

At the end of this section we can prove the following partial characterization of negative $(\omega,m)$-subharmonic functions with bounded $(p,m)$-energy.

\begin{proposition}
Let $X$ be a connected and compact K\"ahler manifold of complex dimension $n$, where $\omega$ is a K\"ahler form on $X$ such that $\int_{X}\omega^n=1$. Also let $1\leq m\leq n$ and $p>0$. Then
\[
\left\{u\in \mathcal {SH}^{-}_m(X,\omega): \int_0^{\infty}t^{m+p-1}\operatorname{cap}_m(\{u<-t\})\,dt<\infty\right\}\subset \mathcal E_m^p(X,\omega).
\]
In particular, if for $u\in \mathcal {SH}^{-}_m(X,\omega)$ there exist constants $C(u)>0$ and $\epsilon >0$ such that
\[
\operatorname{cap}_m(\{u<-t\})\leq \frac {C(u)}{t^{p+m+\epsilon}},
\]
then $u\in \mathcal E_m^p(X,\omega)$.
\end{proposition}

\begin{proof}
Let $u\in \mathcal {SH}^{-}_m(X,\omega)$, and assume that
\[
\int_0^{\infty}t^{m+p-1}\operatorname{cap}_m(\{u<-t\})\,dt<\infty\, .
\]
Then without lost of generality we can assume that $u\leq -1$. Let us define $u_t=\max (u,-t)$, $t\geq 1$, then $v=\frac {u_t}{t}\in \mathcal {SH}_m(X,\omega)$, $-1\leq v\leq 0$, and
\[
\omega_v^m\wedge \omega^{n-m}\geq t^{-m}\omega_{u_t}^m\wedge \omega^{n-m}.
\]
We then have by Proposition~\ref{elementary}
\begin{multline*}
(\omega_{u_t}^m\wedge \omega^{n-m})(\{u<-t\})\leq t^m(\omega_v^m\wedge \omega^{n-m})(\{u<-t\}) \\ \leq t^m\operatorname{cap}_m(\{u<-t\})\to 0, \text{ as } t\to \infty\, .
\end{multline*}
Thus,  $u\in \mathcal E_m(X,\omega)$. Furthermore,
\begin{multline*}
(\omega_{u}^m\wedge \omega^{n-m})(\{u\leq-t\})=\int_X\omega^n-(\omega_{u}^m\wedge \omega^{n-m})(\{u>-t\})\\
=\int_X\omega^n-(\omega_{u_t}^m\wedge \omega^{n-m})(\{u>-t\})=(\omega_{u_t}^m\wedge \omega^{n-m})(\{u\leq-t\}).
\end{multline*}
Finally,
\begin{multline*}
\int_X(-u)^p\operatorname{H}_m(u)=p\int_1^{\infty}t^{p-1}(\omega_{u}^m\wedge \omega^{n-m})(\{u<-t\})dt\\
\leq p\int_1^{\infty}t^{p-1}(\omega_{u_t}^m\wedge \omega^{n-m})(\{u\leq-t\})dt\leq p \int_1^{\infty}t^{m+p-1}\operatorname{cap}_m(\{u\leq-t\})\,dt<\infty.
\end{multline*}
\end{proof}

\section{The complex Hessian equations}

In this section we consider complex Hessian equations for $\mathcal E_m^p(X,\omega)$. We need the following generalization of Theorem~\ref{sobolev}.

\begin{theorem}\label{lemmaep} Let $n\geq 2$, $p>0$, and let $1\leq m\leq n$. Assume that $(X,\omega)$ is a connected and compact K\"ahler manifold of complex dimension $n$, where $\omega$ is a K\"ahler form on $X$ such that $\int_{X}\omega^n=1$. Furthermore, assume that $\mu$ is a Borel measure defined on $X$. Fix a constant $\beta$ such that $1>\beta>\max\left(\frac {pn-n}{pn-n+m},\frac {p}{p+1}\right)$, for $p>1$, and $\beta=\frac {p}{p+1}$ for $p\leq 1$. The following conditions are then equivalent:
\begin{enumerate}\itemsep2mm
\item $\mathcal E_m^p(X,\omega)\subset L^q(X,\mu)$;
\item there exists a constant $C>0$ such that for all $u\in \mathcal E_m(X,\omega)\cap L^{\infty}(X)$ with $\sup_Xu=-1$ it holds
\[
\int_X(-u)^q\,d\mu\leq Ce_{p,m}(u)^{\frac {q\beta}{p}};
\]
\item there exists a constant $C>0$ such that for all $u\in \mathcal E_m^p(X,\omega)$ with $\sup_Xu=-1$ it holds
\[
\int_X(-u)^q\,d\mu\leq Ce_{p,m}(u)^{\frac {q\beta}{p}}.
\]
\end{enumerate}
\end{theorem}
\begin{proof}
The implication (2)$\Rightarrow$(1) is obvious. The equivalence (2)$\Leftrightarrow$(3) follows by approximation. We shall prove (1)$\Rightarrow$(2).

Assume first that $p>1$. To prove this implication assume that condition (2) is not true, i.e., there exists a sequence $u_j\in \mathcal E_m(X,\omega)\cap L^{\infty}(X)$, $\sup_Xu_j=-1$, such that
\begin{equation}\label{6}
\int_X(-u_j)^q\,d\mu\geq 4^{jq}e_{p,m}(u_j)^{\frac {q\beta}{p}}.
\end{equation}

\emph{Case 1.} If the sequence $\{e_{p,m}(u_j)\}$ is bounded (or it contains a bounded subsequence), then let us define
\[
u=\sum_{j=1}^{\infty}\frac {1}{2^j}u_j.
\]
Then $u$ belongs to $\mathcal {SH}_m(X,\omega)$, and by Lemma~\ref{lem7} it follows
\[
\int_X(-u)^p\operatorname{H}_m(u)\leq C(p,m)\sup_{j\in \mathbb N}e_{p,m}(u_j)<\infty .
\]
Hence, $u\in \mathcal E_m^p(X,\omega)$. On the other hand by (\ref{6})
\[
\int_X(-u)^q\,d\mu\geq \frac {1}{2^{jq}}\int_X(-u_j)^q\,d\mu\geq \frac {1}{2^{jq}}4^{jq}e_{p,m}(u_j)^{\frac{q\beta}{p}}\geq 2^{jq}\to \infty, \qquad \text{ as } j\to \infty.
\]
Thus,  $u\notin L^q(X,\mu)$.

\emph{Case 2.} Now assume that $e_{p,m}(u_j)\to \infty$. Let us define $v_j=t_ju_j$, where
\begin{equation}\label{t}
t_j=e_{p,m}(u_j)^{-\frac {\beta}{p}}.
\end{equation}
Then we have by Theorem~\ref{sobolev} and Lemma~\ref{lem7}
\begin{multline*}
e_{p,m}(v_j)=t_j^p\int_X(-u_j)^p\omega^n+t_j^p\sum_{k=1}^{m}\binom {m}{k}t_j^k(1-t_j)^{m-k}\int_X(-u_j)^p\omega_{u_j}^k\wedge \omega^{n-k}\\
\leq t_j^p\int_X(-u_j)^p\omega^n+C2^mt_j^{p+1}e_{p,m}(u_j)\\
\leq C't_j^pe_{p,m}(u_j)^{\beta}+C2^mt_j^{p+1}e_{p,m}(u_j) <+\infty.
\end{multline*}
Therefore, we can repeat the argument from the first case to show that function
\[
v=\sum_{j=1}^{\infty}\frac {1}{2^j}v_j
\]
belongs to $\mathcal {SH}^p_m(X,\omega)$, but $v\notin L^q(X,\mu)$, since
\begin{multline*}
\int_X(-v)^q\,d\mu\geq \frac {1}{2^{jq}}t_j^q\int_X(-u_j)^q\,d\mu \\
\geq \frac {1}{2^{jq}}4^{jq}t_j^qe_{p,m}(u_j)^{\frac{q\beta}{p}}=2^{jq}\to \infty, \text{ as } j\to \infty.
\end{multline*}

Next, assume that  $p\leq 1$ and $\beta=\frac {p}{p+1}$. By~\cite{L2} it follows that if $u$ is $(\omega,m)$-subharmonic function such that $\sup_Xu=-1$, then there exists a constant $C'$ which does not depending on $u$ such that
\[
\int_X(-u)^p\omega^n\leq\left(\int_X(-u)\omega^n\right)^{p}\leq (C')^p,
\]
and then we repeat the above proof for the case when $p>1$.
\end{proof}

By making the best use of Theorem~\ref{lemmaep} we prove the following theorem. Theorem~\ref{Dirichlet} was in the case $p=1$ proved in~\cite{LN}.

\begin{theorem}\label{Dirichlet} Let $n\geq 2$, $p>0$, and let $1\leq m\leq n$. Assume that $(X,\omega)$ is a connected and compact K\"ahler manifold of complex dimension $n$, where $\omega$ is a K\"ahler form on $X$ such that $\int_{X}\omega^n=1$. Furthermore, assume that $\mu$ is a Borel probability measure defined on $X$. The following conditions
are then equivalent:
\begin{enumerate}\itemsep2mm
\item $\mathcal E_m^p(X,\omega)\subset L^p(X,\mu)$;
\item there exists unique $(\omega,m)$-subharmonic function $u$ in $\mathcal E_m^p(X,\omega)$ such that $\sup_X u=-1$ and $\operatorname{H}_m(u)=\mu$.
\end{enumerate}
\end{theorem}
\begin{proof}
Implication (2)$\Rightarrow$(1) follows from Lemma~\ref{lem7}. Next, we shall prove implication (1)$\Rightarrow$(2). To do so let us define the following collection of
Borel probability measures
\[
\mathcal M=\left\{\mu: \mu(X)=1, \mu(K)\leq \operatorname{cap}_m(K), K\subset X\right\}.
\]
It was proved in~\cite{LN} that $\mathcal M$ is convex and compact. Furthermore, for any Borel probability $\mu$ measure we have the following decomposition
\[
\mu=f\nu+\sigma, \ \text{ where } \ \nu\in \mathcal M, \ \sigma\bot \mathcal M, \ f\in L^1(\nu).
\]
If we assume that $\mu$ vanishes on $m$-polar sets, then $\sigma=0$. By assumption $\mu$ is a Borel probability measure defined on $X$ such that $\mathcal E_m^p(X,\omega)\subset L^p(X,\mu)$. Thus, $\mu$ vanishes on $m$-polar sets, so there exist $\nu\in \mathcal M$ and $f\in L^1(\nu)$ such that $\mu=f\nu$. Set
\[
\mu_j=c_j\min(f,j)\nu,
\]
where $c_j>0$ is such that $\mu_j(X)=1$. It follows from~\cite{LN} that there exists $u_j\in \mathcal E_m(X,\omega)$ such that $\operatorname{H}_m(u_j)=\mu_j$ and $\sup_Xu_j=-1$. Without loss of generality we can assume that $u_j\to u$ in $L^1(X)$. Next, we define
\[
u_{j,k}=\max (u_j,-k)\in L^{\infty}(X)\, .
\]
This construction implies that $u_{j,k}\in \mathcal E_m^p(X,\omega)$, and $u_{j,k}\searrow u_j$, as $k\to \infty$. Hence, $e_{p,m}(u_{j,k})\to e_{p,m}(u_j)$, $k\to \infty$. By Theorem~\ref{lemmaep} it follows for some $\beta<1$ that
\[
\int_X(-u_{j,k})^p\operatorname{H}_m(u_j)\leq \int_X(-u_{j,k})^p\,d\mu_j\leq C\left(\int_X(-u_{j,k})^p\operatorname{H}_m(u_{j,k})\right)^{\beta},
\]
and since we have $\int_X(-u_{j,k})^p\operatorname{H}_m(u_j)\to e_{p,m}(u_j)$ and $e_{p,m}(u_{j,k})\to e_{p,m}(u_j)$, as $k\to \infty$, we get
\[
\sup_ke_{p,m}(u_{j,k})<\infty.
\]
Thus, $u_j\in \mathcal E_m^p(X,\omega)$. Theorem~\ref{lemmaep} yields, again for some $\beta<1$, that
 \[
\int_X(-u_{j})^p\operatorname{H}_m(u_j)\leq c_j\int_X(-u_{j})^p\,d\mu\leq c_jC\left(\int_X(-u_{j})^p\operatorname{H}_m(u_{j})\right)^{\beta}.
\]
Thus, $\sup_je_{p,m}(u_{j})<\infty$. Let us define $v_j=\left(\sup_{k\geq j}u_k\right)^*$. Here $(\,)^*$ denotes the upper semicontinuous regularization. Then $v_j$ is a decreasing sequence of function from $\mathcal E_m^p(X,\omega)$, $v_j\searrow u$, $j\to \infty$. Furthermore, since $\sup_je_{p,m}(v_{j})<\infty$,  then we can conclude that $u\in\mathcal E_m^p(X,\omega)$. Then by~\cite{LN} we conclude that $\operatorname{H}_m(v_j)\geq \min(f,j)\nu$, after passing to the limit with $j$ we get $\operatorname{H}_m(u)\geq \mu$, but since both measure $\operatorname{H}_m(u)$ and $\mu$ have the same total mass we conclude that $\operatorname{H}_m(u)=\mu$.
\end{proof}

At the end of this section we shall prove the following proposition.

\begin{proposition}
Assume the same conditions as in Theorem~\ref{Dirichlet}.
\begin{enumerate}\itemsep2mm
\item[$a)$] If there exist constants $\alpha>\frac {p}{p+1}$  and $C>0$ such that for all Borel sets $E$ it holds
\[
\mu(E)\leq C\operatorname{cap}_m(E)^{\alpha},
\]
then $\mathcal E_m^p(X,\omega)\subset L^p(X,\mu)$.
\item[$b)$] If $\mathcal E_m^p(X,\omega)\subset L^p(X,\mu)$,  then for fixed $\beta>\max\left(\frac {pn-n}{pn-n+m},\frac {p}{p+1}\right)$, if $p>1$, and $\beta=\frac {p}{p+1}$ if $p\leq 1$, there exists a constant $C>0$ such that for all Borel sets $E$ it holds
\[
\mu(E)\leq C\operatorname{cap}_m(E)^{\beta}.
\]
\end{enumerate}
\end{proposition}
\begin{proof}
\emph{a)} Let $u\in \mathcal E_m^p(X,\omega)$ with $\sup_X u=-1$. From Theorem~\ref{thm_lq} it follows that
\begin{multline*}
\int_X(-u)^p\,d\mu=p\int_1^{\infty}t^{p-1}\mu(\{u<-t\})dt\\
\leq pC\int_1^{\infty}t^{p-1}\operatorname{cap}_m(\{u<-t\})dt
\leq pC\int_1^{\infty}t^{p-1}\left(\frac {C'(u)}{t^{p+1}}\right)^{\alpha}dt\\ =pC(C'(u))^{\alpha}\int_1^{\infty}t^{p-1-\alpha p-\alpha}dt<\infty.
\end{multline*}

\emph{b)} Assume that $\mathcal E_m^p(X,\omega)\subset L^p(X,\mu)$. From Theorem~\ref{lemmaep} it follows that there exists a constant $C>0$ such that for all $v\in \mathcal E_m^p(X,\omega)\cap L^{\infty}(X)$, with $\sup_Xv=-1$, it holds
\begin{equation}\label{13}
\int_X(-v)^p\,d\mu\leq Ce_{p,m}(v)^{\beta}.
\end{equation}
Let $E$ be a Borel set, and let $h_{m,E}$ be the $m$-extremal function for the set $E$, i.e.
\[
h_{m,E}=\big(\sup\left\{u\in \mathcal {SH}_m(X,\omega): u\leq -1 \ \text { on } \ E, u\leq 0 \ \text { on } \ X\right\}\big)^*,
\]
where $(\,)^*$ denotes the upper semicontinuous regularization (see e.g.~\cite[Section~4]{LN} for further information). Using $h_{m,E}$ in (\ref{13}) we arrive at
\[
\mu(E)\leq \int_X(-h_{m,E})^p\,d\mu\leq C e_{p,m}(h_{m,E})^{\beta}=C\operatorname{cap}_m(E)^{\beta}.
\]
\end{proof}

\end{document}